\documentclass[10pt,draft]{article}
\usepackage{amssymb,amsmath,amsthm,latexsym}

\newtheorem{theorem}{Theorem}
\newtheorem{lemma}{Lemma}
\newtheorem{corollary}{Corollary}

\numberwithin{equation}{section}

\newcommand{\R}{{\mathbb{R}}}

\newcommand{\mF}{{\mathcal{F}}}

\newcommand{\gG}{{\mathit{\Gamma}}}

\newcommand{\Bd}{\mathrm{bd\,}}
\newcommand{\Cl}{\mathrm{cl\,}}

\newcommand{\Dim}{\mathrm{dim\,}}
\newcommand{\Int}{\mathrm{int\,}}
\newcommand{\Rbd}{\mathrm{rbd\,}}
\newcommand{\Rec}{\mathrm{rec\,}}
\newcommand{\Rint}{\mathrm{rint\,}}

\begin{document}

\begin{center}

{\Large \textbf{Convex Quadric Surfaces}}

\bigskip

Valeriu Soltan

\medskip

Department of Mathematical Sciences, George Mason University

Fairfax, Virginia 22030, USA

e-mail: vsoltan@gmu.edu

\bigskip

\begin{quote}

\noindent \textbf{Abstract.} We describe convex quadric surfaces
in $\R^n$ and characterize them as convex surfaces with quadric
sections by a continuous family of hyperplanes.

\smallskip


\smallskip


\end{quote}

\end{center}

\section{Introduction and main results}

\bigskip

\noindent Characterizations of ellipses and ellipsoids among
convex bodes in the plane or in space became an established topic
of convex geometry on the turn of twenty's century. Comprehensive
surveys on various characteristic properties of ellipsoids in the
Euclidean space $\R^n$ are given in~\cite{g-h76} and~\cite{pet83}
(see also~\cite{h-m93}). Similar characterizations of unbounded
convex quadrics, like paraboloids, sheets of elliptic hyperboloids
or elliptic cones, are given by a short list of sporadic results
(see, e.\,g., \cite{ait74,ale39,sol08,sol09}). Furthermore, even a
classification of convex quadrics in $\R^n$ is not established
(although it is used in \cite{sol08,sol09} without proof). Our
goal here is to describe convex quadrics in $\R^n$ and to provide
a characteristic property of these surfaces in terms of hyperplane
sections.

We recall that a \textit{quadric surface} (or a \textit{second
degree surface}) in $\R^n$, $n \ge 2$, is the locus of points $x =
(\xi_1, \ldots, \xi_n)$ that satisfy a quadratic equation
\begin{equation} \label{quadric-surface1}
\sum_{i,k=1}^n a_{ik} \xi_i \xi_k + 2 \sum_{i=1}^n b_i \xi_i + c =
0,
\end{equation}
where not all $a_{ik}$ are zero. A \textit{convex surface} is the
boundary of an $n$-dimensional convex set distinct from $\R^n$. In
particular, a hyperplane and a pair of parallel hyperplanes are
convex surfaces. We say that a convex surface $S \subset \R^n$ is
a \textit{convex quadric} provided there is a real quadric surface
$Q \subset \R^n$ and a convex component $U$ of $\R^n \setminus Q$
such that $S$ is the boundary of $U$. The following theorem plays
a key role in the description of convex quadrics.

\begin{theorem} \label{quadric-complement-conv}
The complement of a real quadric surface $Q \subset \R^n$, $n \ge
2$, is the union of four or fewer open sets; at least one of these
sets is convex if and only if the canonical form of $Q$ is given
by one of the equations
\begin{align*}
& a_1 \xi_1^2 + \cdots + a_k \xi_k^2 = 1, \ \ 1 \le k \le n,
\phantom{aaaaaaaaaaaaaaaaa} \\
& a_1 \xi_1^2 - a_2 \xi_2^2 - \cdots - a_k \xi_k^2 = 1, \ \ 2 \le
k \le
n,  \\
& a_1 \xi_1^2 = 0, \\
& a_1 \xi_1^2 - a_2 \xi_2^2 - \cdots - a_k \xi_k^2 = 0, \ \ 2 \le
k \le
n, \\
& a_1 \xi_1^2 + \cdots + a_{k-1} \xi_{k-1}^2 = \xi_k, \ \ 2 \le k
\le n,
\end{align*}
where all scalars $a_i$ involved are positive.
\end{theorem}

\begin{corollary} \label{convex-quadrics}
A convex surface $S \subset \R^n$, $n \ge 2$, is a convex quadric
if and only if there are Cartesian coordinates $\xi_1, \dots,
\xi_n$ such that $S$ can be expressed by one of the equations
\begin{align*}
& a_1 \xi_1^2 + \cdots + a_k \xi_k^2 = 1, \ \ 1 \le k \le n,
\phantom{aaaaaaaaaaaaaaaaa} \\
& a_1 \xi_1^2 - a_2 \xi_2^2 - \cdots - a_k \xi_k^2 = 1, \
\xi_1 \ge 0, \ \ 2 \le k \le n, \\
& a_1 \xi_1^2 = 0, \\
& a_1 \xi_1^2 - a_2 \xi_2^2 - \cdots - a_k \xi_k^2 = 0, \ \xi_1
\ge 0, \ \
2 \le k \le n, \\
& a_1 \xi_1^2 + \cdots + a_{k-1} \xi_{k-1}^2 = \xi_k, \ \ 2 \le k
\le n,
\end{align*}
where all scalars $a_i$ involved are positive.
\end{corollary}

In what follows, a \textit{plane} of dimension $m$ is a
translation of an $m$-dimensional subspace. We say that a plane
$L$ \textit{properly} intersects an $n$-dimensional convex set $K$
provided $L$ intersects both the boundary $\Bd K$ and the interior
$\Int K$ of $K$.

A well-known result of convex geometry states that the boundary of
a convex body $K \subset \R^n$ is an ellipsoid if and only if
there is a point $p \in \Int K$ such that all sections of $\Bd K$
by 2-dimensional planes through $p$ are ellipses
(see~\cite{a-m-u35,kub14} for $n = 3$
and~\cite[pp.\,91--92]{bus55} for $n \ge 3$). This result is
generalized in \cite{sol08} by showing that the boundary of an
$n$-dimensional closed convex set $K \subset \R^n$ is a convex
quadric if and only if there is a point $p \in \Int K$ such that
all sections of $\Bd K$ by 2-dimensional planes through $p$ are
convex quadric curves. In this regard, we pose the following
problem (solved in~\cite{bur76,hob74} for the case of convex
bodies): \textit{Given an $n$-dimensional closed convex set $K
\subset \R^n$ distinct from $\R^n$, $n \ge 3$, and a point $p \in
\R^n$, is it true that either $\Bd K$ is a convex quadric or $K$
is a convex cone with apex $p$ provided all proper sections of
$\Bd K$ by 2-dimensional planes through $p$ are convex quadric
curves?}

Kubota~\cite{kub14} proved that, given a pair of bounded convex
surfaces in $\R^3$, one being enclosed by the other, if all planar
sections of the biggest surface by planes tangent to the second
surface are ellipses, then the biggest surface is an ellipsoid.
Independently, Bianchi and Gruber~\cite{b-g87} established the
following far-reaching assertion: If $K$ is a convex body in
$\R^n$, $n \ge 3$, and $\delta (u)$ is a continuous real-valued
function on the unit sphere $S^{n-1} \subset \R^n$ such that for
each vector $u \in S^{n-1}$ the hyperplane $H(u) = \{ x \mid x \!
\cdot \! u = \delta(u) \}$ intersects $\Bd K$ along an $(n -
1)$-dimensional ellipsoid, then $\Bd K$ is an ellipsoid. Our
second theorem extends this assertion to the case of
$n$-dimensional closed convex sets.

\begin{theorem} \label{convex-quadric-section}
Let $K \subset \R^n$ be an $n$-dimensional closed convex set
distinct from $\R^n$, $n \ge 3$, and $\delta (u)$ be a continuous
real-valued function on the unit sphere $S^{n-1} \subset \R^n$
such that for each vector $u \in S^{n-1}$ the hyperplane $H(u) =
\{ x \mid x \! \cdot \! u = \delta(u) \}$ either lies in $K$ or
intersects $\Bd K$ along an $(n - 1)$-dimensional convex quadric.
Then $\Bd K$ is a convex quadric.
\end{theorem}

In what follows, the origin of $\R^n$ is denoted $o$. We say that
a plane $L$ \textit{supports} a closed convex set $K$ provided $L$
intersects $K$ such that $L \cap \Int K = \varnothing$. The
\textit{recession cone} of $K$ is defined by
\[
\Rec K = \{ y \in \R^n \mid x + \alpha y \in K \ \hbox{for all} \
x \in K \ \mathrm{and} \ \alpha \ge 0\}.
\]
It is well-known that $\Rec K \ne \{o\}$ if and only if $K$ is
unbounded; $K$ is called \textit{line-free} if it contains no
line. Finally, $\Rint M$ and $\Rbd M$ denote the relative interior
and the relative boundary of a convex set $M \subset \R^n$.

\section{Auxiliary Lemmas}

The proof of Theorem~\ref{quadric-complement-conv} uses a
description of certain quadric surfaces in $\R^n$ as consecutive
re\-vo\-lutions of lower-dimensional quadrics. To describe these
revolutions, choose any subspaces $L_1, L_2$, and $L_3$ of $\R^n$
such that $L_1 \subset L_2 \subset L_3$ and
\[
\Dim L_1 = m - 1, \ \ \Dim L_2 = m, \ \ \Dim L_3 = m + 1, \ \ 2
\le m \le n - 1.
\]
Let $M$ be the 2-dimensional subspace of $L_3$ orthogonal to
$L_1$. Given a point $y \in L_2$, put $M_y = y + M$ and denote by
$z$ the point of intersection of $L_1$ and $M_y$ (obviously, $z$
is the orthogonal projection of $y$ on $L_1$). Let $C_y$ be the
circumference in $M_y$ with center $z$ and radius $\| y - z \|$.
We say that a set $X \subset L_3$ is the \textit{revolution} of a
set $Y \subset L_2$ about $L_1$ within $L_3$ provided $X = \cup \{
C_y \mid y \in Y \}$. A set $Z \subset \R^n$ is called
\textit{symmetric} about a subspace $N \subset \R^n$ provided for
any point $x \in Z$ and its orthogonal projection $u$ on $N$, the
point $2 u - x$ lies in $Z$. In these terms, we formulate three
lemmas, the first one being obvious.

\begin{lemma} \label{symm-component}
If a set $Y \subset L_2$ is symmetric about $L_1$ and $X$ is the
revolution of\, $Y$ about $L_1$ within $L_3$, then $X$ is
symmetric about $L_2$ and any component of $X$ is the revolution
of a suitable component of\, $Y$ about $L_1$ within $L_3$. \qed
\end{lemma}

In what follows, $\langle e_1, \dots, e_k \rangle$ means the span
of vectors $e_1, \dots, e_k \in \R^n$.

\begin{lemma} \label{symm-convex}
If a set $Y \subset L_2$ is symmetric about $L_1$ and $X$ is the
revolution of\, $Y$ about $L_1$ within $L_3$, then $X$ is a convex
set if and only if\, $Y$ is a convex set.
\end{lemma}

\begin{proof}
Without loss of generality, we may put $L_3 = \R^n$. Choose an
orthonormal basis $e_1, \dots, e_n$ for $\R^n$ such that
\[
L_1 = \langle e_1, \dots, e_{n-2} \rangle \quad \text{and} \quad
L_2 = \langle e_1, \dots, e_{n-1} \rangle.
\]
Clearly, $x = (\xi_1, \dots, \xi_n)$ belongs to $X$ if and only if
there is a point
\[
y = (\xi_1, \dots, \xi_{n-2}, \xi'_{n-1}, 0) \in Y \quad
\text{where} \quad \xi'_{n-1} = \sqrt{ \xi_{n-1}^2 + \xi_n^2}.
\]

If $X$ is convex, then $Y$ is convex due to $Y = X \cap L_2$. Let
$Y$ be convex. Choose any points $a = (\alpha_1, \dots, \alpha_n)$
and $b = (\beta_1, \dots, \beta_n)$ in $X$ and a scalar $\lambda
\in [0, 1]$. We intend to show that $c = (1 - \lambda) a + \lambda
b \in X$. Let
\[
a' = (\alpha_1, \dots, \alpha_{n-2}, \alpha_{n-1}', 0), \quad b' =
(\beta_1, \dots, \beta_{n-2}, \beta_{n-1}', 0),
\]
and
\[
c' = ((1 - \lambda) \alpha_1 + \lambda \beta_1, \dots, (1 -
\lambda) \alpha_{n-2} + \lambda \beta_{n-2}, (1 - \lambda)
\alpha_{n-1}' + \lambda \beta_{n-1}', 0)
\]
be points in $Y$ where
\[
\alpha_{n-1}' = \sqrt{\alpha_{n-1}^2 + \alpha_n^2}, \quad
\text{and} \quad \beta_{n-1}' = \sqrt{\beta_{n-1}^2 + \beta_n^2}.
\]
Then $a', b' \in Y$, and $c' = (1 - \lambda) a' + \lambda b' \in
Y$ due to convexity of $Y$. Because $Y$ is symmetric about $L_1$,
we have
\[
((1 - \lambda) \alpha_1 + \lambda \beta_1, \dots, (1 - \lambda)
\alpha_{n-2} + \lambda \beta_{n-2}, \mu, 0) \in Y
\]
for any scalar $\mu$ with $|\mu| \le (1 - \lambda) \alpha_{n-1}' +
\lambda \beta_{n-1}'$. Let
\[
y = ((1 - \lambda) \alpha_1 + \lambda \beta_1, \dots, (1 -
\lambda) \alpha_{n-2} + \lambda \beta_{n-2}, \rho, 0),
\]
where
\[
\rho = \sqrt{ \big( (1 - \lambda) \alpha_{n-1} + \lambda
\beta_{n-1} \big)^2 + \big( (1 - \lambda) \alpha_n + \lambda
\beta_n \big)^2}.
\]
From $\alpha_{n-1} \beta_{n-1} + \alpha_n \beta_n \le
\alpha_{n-1}' \beta_{n-1}'$,  we obtain $\rho \le (1 - \lambda)
\alpha_{n-1}' + \lambda \beta_{n-1}'$, which gives $y \in Y$.
Clearly, the point
\[
z = ((1 - \lambda) \alpha_1 + \lambda \beta_1, \dots, (1 -
\lambda) \alpha_{n-2} + \lambda \beta_{n-2}, 0, 0)
\]
is the orthogonal projection of $y$ on $L_1$. The equalities $\| c
- z \| = \| y - z \| = \rho$ imply that $c \in C_y \subset X$.
Hence $X$ is convex.
\end{proof}

Let $Q \subset \R^n$ be a real quadric surface. A suitable choice
of Cartesian coordinates transforms \eqref{quadric-surface1} into
one of the following canonical forms
\begin{align*}
A_k \ \, &: \ \xi_1^2 + \cdots + \xi_k^2 = 1, \ 1 \le k \le n,
\\
B_{k,r} &: \ \xi_1^2 + \cdots + \xi_k^2 - \xi_{k+1}^2 - \cdots -
\xi_r^2 = 1, \ 1 \le k < r \le n,  \\
C_k \ \, &: \ \xi_1^2 + \cdots + \xi_k^2 = 0,
\ 1 \le k \le n, \\
D_{k,r} &: \ \xi_1^2 + \cdots + \xi_k^2 -
\xi_{k+1}^2 - \cdots - \xi_r^2 = 0, \ 1 \le k < r \le n,  \\
E_{k,r} &: \ \xi_1^2 + \cdots + \xi_k^2 - \xi_{k+1}^2 - \cdots -
\xi_{r-1}^2 = \xi_r, \ 1 \le k < r \le n.
\end{align*}

\begin{lemma} \label{revolution-cases}
Within $\R^n$, $n \ge 3$,
\begin{itemize}

\item[$1)$] $A_n$ is the revolution of $A_{n-1} \subset
\langle e_1, \dots, e_{n-1} \rangle$ about $\langle e_1, \dots,
e_{n-2} \rangle$,

\item[$2)$] $B_{k,n}$ is the revolution of $B_{k,n-1} \subset
\langle e_1, \dots, e_{n-1} \rangle$ about $\langle e_1, \dots,
e_{n-2} \rangle$, $1 \le k \le n - 2$,

\item[$3)$] $D_{k,n}$ is the revolution of $D_{k,n-1} \subset
\langle e_1, \dots, e_{n-1} \rangle$ about $\langle e_1, \dots,
e_{n-2} \rangle$, $1 \le k \le n - 2$,

\item[$4)$] $B_{k,n}$ is the revolution of $B_{k-1,n-1} \subset
\langle e_2, \dots, e_n \rangle$ about $\langle e_3, \dots, e_n
\rangle$, $2 \le k \le n - 1$,

\item[$5)$] $D_{k,n}$ is the revolution of $D_{k-1,n-1} \subset
\langle e_2, \dots, e_n \rangle$ about $\langle e_3, \dots, e_n
\rangle$, $2 \le k \le n - 1$.

\end{itemize}
\end{lemma}

\begin{proof}
$1)$ Given a point $x = (\xi_1, \dots, \xi_n) \in A_n$, put
\begin{align} \label{eqn:revolution-cases}
y = (\xi_1, \dots, \xi_{n-2}, \sqrt{\xi_{n-1}^2 + \xi_n^2}, 0),
\quad z = (\xi_1, \dots, \xi_{n-2}, 0, 0).
\end{align}
Then $y \in A_{n-1} \subset \langle e_1, \dots, e_{n-1} \rangle$
and $z$ is the orthogonal projection of $y$ on $\langle e_1,
\dots, e_{n-2} \rangle$. From
\[
\| x - z \| = \| y - z \| = \sqrt{\xi_{n-1}^2 + \xi_n^2}
\]
we see that $x \in C_y$. So, $A_n$ lies in the revolution of
$A_{n-1}$ about $\langle e_1, \dots, e_{n-2} \rangle$. Conversely,
if $y = (\eta_1, \dots, \eta_{n-1}, 0)$ is a point in $A_{n-1}
\subset \langle e_1, \dots, e_{n-1} \rangle$ and $z = (\eta_1,
\dots, \eta_{n-2}, 0, 0)$ is the orthogonal projection of $y$ on
$\langle e_1, \dots, e_{n-2} \rangle$, then any point $u$ from the
circle $C_y \subset y + \langle e_{n-1}, e_n \rangle$ can be
written as
\[
u = (\eta_1, \dots, \eta_{n-2}, \gamma_{n-1}, \gamma_n), \quad
\text{where} \quad \gamma_{n-1}^2 + \gamma_n^2 = \eta_{n-1}^2.
\]
Clearly,  $u \in A_n$, which shows that $A_n$ contains the
revolution of $A_{n-1}$ about $\langle e_1, \dots, e_{n-2}
\rangle$.

Cases $2)$--$5)$ are considered similarly, where the points $y$
and $z$ are defined, respectively, by (\ref{eqn:revolution-cases})
in cases $2)$ and $3)$, and by
\[
y = (0, \sqrt{\xi_1^2 + \xi_2^2}, \xi_3, \dots, \xi_n), \quad z =
(0, 0, \xi_3, \dots, \xi_n)
\]
in cases $4)$ and $5)$.
\end{proof}

\section{Proof of Theorem~\ref{quadric-complement-conv}}

Let $Q \subset \R^n$ be a real quadric surface. We may suppose
that $Q$ has one of the forms $A_k$, $B_{k,r}$, $C_k$, $D_{k,r}$,
$E_{k,r}$ described above. First, we exclude the trivial cases $Q
= A_1$ (when $Q$ is a pair of parallel hyperplanes) and $Q = C_k$
(when $Q$ is an $(n - k)$-dimensional subspace). Furthermore, we
can reduce the proof to the case when $Q$ is has one of the forms
$A_n, B_{k,n}, D_{k,n}, E_{k,n}$. Indeed, if $k < n$ or $r < n$,
then $Q$ is a both-way unbounded cylinder, that is, it the
Cartesian product of a subspace $\R^k$ (respectively, $\R^r$) and
a quadric $P$ of the same type in the orthogonal complement
$\R^{n-k}$ (respectively, $\R^{n-r}$); clearly, $Q$ satisfies the
conclusion of the theorem if and only if $P$ does.

Our further consideration is organized by induction on $n$. The
cases $n = 2$ and $n = 3$ follow immediately from the well-known
properties of quadric curves and surfaces. Suppose that $n \ge 4$.
Assuming that the conclusion of
Theorem~\ref{quadric-complement-conv} holds for all $m < n$, let
the quadric surface $Q \subset \R^n$ have one of the forms $A_n,
B_{k,n}, D_{k,n}, E_{k,n}$. We consider these forms separately.

\textit{Case}~1. Let $Q = A_n$. By Lemma~\ref{revolution-cases},
$A_n$ can be obtained from
\[
A_2 = \{ (\xi_1, \xi_2) \mid \xi_1^2 + \xi_2^2 = 1 \} \subset
\langle e_1, e_2 \rangle
\]
by consecutive revolutions of $A_i \subset \langle e_1, \dots, e_i
\rangle$ about $\langle e_1, \dots, e_{i-1} \rangle$ within the
subspace $\langle e_1, \dots, e_{i+1} \rangle$, $i = 2, \dots, n -
1$. Since both components of $\langle e_1, e_2 \rangle \setminus
A_2$ are symmetric about the line $\langle e_1 \rangle$,
Lemmas~\ref{symm-component} and~\ref{symm-convex} imply that $\R^n
\setminus A_n$ consists of two components; one of them, given by
$\xi_1^2 + \dots + \xi_n^2 < 1$, is convex.

\textit{Case}~ 2. Let $Q = B_{k,n}$, $1 \le k \le n - 1$. If $k =
1$, then Lemma~\ref{revolution-cases} implies that $B_{1,n}$ can
be obtained from
\[
B_{1,2} = \{ (\xi_1, \xi_2) \mid \xi_1^2 - \xi_2^2 = 1 \} \subset
\langle e_1, e_2 \rangle
\]
by consecutive revolutions of $B_{1,i} \subset \langle e_1, \dots,
e_i \rangle$ about $\langle e_1, \dots, e_{i-1} \rangle$ within
the subspace $\langle e_1, \dots, e_{i+1} \rangle$, $i = 2, \dots,
n - 1$. Since all three components of $\langle e_1, e_2 \rangle
\setminus B_{1,2}$ are symmetric about the line $\langle e_1
\rangle$, Lemmas~\ref{symm-component} and~\ref{symm-convex} imply
that $\R^n \setminus B_{1,n}$ consists of three components; two of
them, given, respectively, by
\[
\xi_1 > \sqrt{ \xi_2^2 + \dots + \xi_n^2 + 1}  \quad \text{and}
\quad \xi_1 < - \sqrt{ \xi_2^2 + \dots + \xi_n^2 + 1},
\]
are convex. If $k \ge 2$, then $B_{k,n}$ can be obtained from
\[
B_{1,2} = \{ (\xi_k, \xi_{k+1}) \mid \xi_k^2 - \xi_{k+1}^2 = 1 \}
\subset \langle e_k, e_{k+1} \rangle
\]
in two steps. First, we obtain $B_{k,k+1} \subset \R^{k+1} =
\langle e_1, \dots, e_{k+1} \rangle$ by consecutive revolutions of
$B_{i,i+1} \subset \langle e_{k+1-i}, e_{k+2-i}, \dots, e_{k+1}
\rangle$ about $\langle e_{k+2-i}, \dots, e_{k+1} \rangle$ within
$\langle e_{k-i}$, $e_{k+1-i}, \dots, e_{k+1} \rangle$, $i = 1, 2,
\dots, k - 1$. The complement of
\[
B_{2,3} = \{ (\xi_{k-1}, \xi_k, \xi_{k+1}) \mid \xi_{k-1}^2 +
\xi_k^2 - \xi_{k+1}^2 = 1 \}
\]
in $\langle e_{k-1}, e_k, e_{k+1} \rangle$, consists of two
components, both symmetric about $\langle e_k, e_{k+1} \rangle$.
Since none of these components is convex,
Lemmas~\ref{symm-component} and~\ref{symm-convex} imply that
$\R^{k+1} \setminus B_{k,k+1}$ consists of two components, both
symmetric about any $k$-dimensional coordinate subspace of
$\R^{k+1}$, but none of them convex.

Second, we obtain $B_{k,n}$ from $B_{k,k+1}$ by consecutive
revolutions of $B_{k,j} \subset \langle e_1, \dots, e_j \rangle$
about $\langle e_1, \dots, e_{j-1} \rangle$ within $\langle e_1,
\dots, e_{j+1} \rangle$, $j = k + 1, \dots, n - 1$. As above,
$\R^n \setminus B_{k,n}$ consists of two components, none of them
convex.

\textit{Case}~3. Let $Q = D_{k,n}$, $1 \le k \le n - 1$. If $k =
1$, then $D_{1,n}$ can be obtained from
\[
D_{1,2} = \{ (\xi_1, \xi_2) \mid \xi_1^2 - \xi_2^2 = 0 \} \subset
\langle e_1, e_2 \rangle
\]
by consecutive revolutions of $D_{1,i} \subset \langle e_1, \dots,
e_i \rangle$ about $\langle e_1, \dots, e_{i-1} \rangle$ within
the subspace $\langle e_1, \dots, e_{i+1} \rangle$, $i = 2, \dots,
n - 1$. The complement of
\[
D_{1,3} = \{ (\xi_1, \xi_2, \xi_3) \mid \xi_1^2 - \xi_2^2 +
\xi_3^2 = 0 \}
\]
in $\langle e_1, e_2, e_3 \rangle$ consists of tree components,
all symmetric about $\langle e_1, e_2 \rangle$. Since two of these
components are convex, Lemmas~\ref{symm-component}
and~\ref{symm-convex} imply that $\R^n \setminus D_{1,n}$ consists
of three components; two of them, given, respectively, by
\[
\xi_1 > \sqrt{ \xi_2^2 + \dots + \xi_n^2} \quad \text{and} \quad
\xi_1 < - \sqrt{ \xi_2^2 + \dots + \xi_n^2},
\]
are convex.

Since the case $k = n - 1$ is reducible to that of $k = 1$ (by
reordering $e_1, e_2, \dots, e_n$ as $e_n, e_{n-1}, \dots, e_1$),
it remains to assume that $2 \le k \le n - 2$. Then $D_{k,n}$ can
be obtained from
\[
D_{2,3} = \{ (\xi_{k-1}, \xi_k, \xi_{k+1}) \mid \xi_{k-1}^2 +
\xi_k^2 - \xi_{k+1}^2 = 0 \} \subset \langle e_{k-1}, e_k, e_{k+1}
\rangle
\]
in two steps. First, we obtain $D_{2,n-k+2} \subset \langle
e_{k-1}, e_k, \dots, e_n \rangle$ by consecutive revolutions of
$D_{2,i} \subset \langle e_{k-1}, e_k, \dots, e_i \rangle$ about
$\langle e_{k-1}, e_k, \dots, e_{i-1} \rangle$ within $\langle
e_{k-1}$, $e_k, \dots, e_{i+1} \rangle$, $i = k + 1, \dots, n -
1$. Clearly, $\langle e_{k-1}, e_k, e_{k+1} \rangle \setminus
D_{2,3}$ consists of three components; two of them,
\[
\xi_{k+1} > \sqrt{ \xi_{k-1}^2 + \xi_k^2} \quad \text{and} \quad
-\xi_{k+1} < \sqrt{ \xi_{k-1}^2 + \xi_k^2},
\]
are convex and symmetric to each other about $\langle e_{k-1}, e_k
\rangle$. Hence $\langle e_{k-1}, e_k, e_{k+1}$, $e_{k+2} \rangle
\setminus D_{3,4}$ consists of two components, none of them
convex. Lemmas~\ref{symm-component} and~\ref{symm-convex} imply
that $\R^{n-k+2} \setminus D_{2,n-k+2}$ consists of two
components, none of them convex.

Next, we obtain $D_{k,n}$ from $D_{2,n-k+2}$ by consecutive
revolutions of the surface $D_{i,n-k+i} \subset \langle e_{k-i+1},
\dots, e_n \rangle$ about $\langle e_{k-i+2}, \dots, e_n \rangle$
within $\langle e_{k-i}, \dots, e_n \rangle$, $i = 2, \dots, k -
1$. As above, $\R^n \setminus D_{k,n}$ consists of two components,
none of them convex.

\textit{Case}~4. Let $Q = E_{k,n}$, $1 \le k \le n - 1$. Clearly,
$E_{k,n}$ is the graph of a real-valued function $\varphi$ on
$\R^{n-1} = \langle e_1, \dots, e_{n-1} \rangle$, given by
\[
\xi_n = \varphi (\xi_1, \dots, \xi_{n-1}) = \xi_1^2 + \dots +
\xi_k^2 - \xi_{k+1}^2 - \dots - \xi_{n-1}^2.
\]
Hence $\R^n \setminus E_{k,n}$ has two components. The Hessian
$\Big( \frac{\partial^2 \varphi}{\partial \xi_i
\partial \xi_j} \Big)$ is a diagonal $n \times n$-matrix,
with 2's on its first $k$ diagonal positions and $-2$'s on the
other $n - k - 1$ diagonal positions. Therefore, $\varphi$ is not
concave, being convex if and only if $k = n - 1$. So, $\R^n
\setminus E_{k,n}$ has a convex component if and only if $k = n -
1$; it is given by $\xi_1^2 + \cdots + \xi_{n-1}^2 < \xi_n$.

\section{Proof of Theorem~\ref{convex-quadric-section}}

Under the assumption of Theorem~\ref{convex-quadric-section}, we
divide the proof into a sequence of lemmas.

\begin{lemma} \label{section1}
If $K$ contains a line, then $\Bd K$ is a both-way unbounded
convex quadric cylinder.
\end{lemma}

\begin{proof}
If $l$ is a line in $K$, then $K$ is the direct sum $\langle u_0
\rangle \oplus (K \cap H(u_0))$, where $\langle u_0 \rangle$ is
the 1-dimensional subspace spanned by a unit vector $u_0$ parallel
to $l$. By the assumption, $\Bd K \cap H(u_0)$ is an
$(n-1)$-dimensional convex quadric. Hence $\Bd K = \langle u_0
\rangle \oplus (\Bd K \cap H(u_0))$ is a both-way unbounded convex
quadric cylinder.
\end{proof}

Due to Lemma~\ref{section1}, we may further assume that $K$ is
line-free. Then no hyperplane lies in $K$; so, every hyperplane
$H(u)$, $u \in S^{n-1}$, properly intersects $K$.

\begin{lemma} \label{section2}
For any $(n-2)$-dimensional plane $L$ supporting $K$, there is a
hyperplane $H(u)$, $u \in S^{n-1}$, that contains $L$.
\end{lemma}

\begin{proof}
Let $P$ be the 2-dimensional subspace orthogonal to $L$ and $\pi$
be the orthogonal projection of $\R^n$ onto $P$. The intersection
$L \cap P$ is a singleton, say $\{ v \}$. The set $M = \pi (K)$ is
convex, $\Rint M = \pi (\Int K)$, and $v \in \Rbd M$. Denote by
$l$ a line in $P$ that supports $M$ at $v$, and let $u_1, u_2$ be
unit vectors in $P$ such that $u_1$ is parallel to $l$ and $u_2$
is orthogonal to $l$, where $v + u_2$ is an outward unit normal to
$M$ at $v$. Without loss of generality, we may suppose that $u_2$
lies between $u_1$ and $-u_1$ according to the clockwise bypass of
$P \cap S^{n-1}$.

Assume, for contradiction, that no line $l(u) = P \cap H(u)$, $u
\in P \cap S^{n-1}$, contains $v$. We consider the cases $v = o$
and $v \ne o$ separately.

If $v = o$, then $\delta(u) \ne 0$ for all $u \in P \cap S^{n-1}$.
If $\delta (u_1) > 0$, then the continuous curve $\delta (u)  u$
with endpoints $\delta(u_1)  u_1$ and $\delta(-u_1) (-u_1)$,
obtained by the clockwise bypass of $P \cap S^{n-1}$, entirely
lies in the closed halfplane of $P$ bounded by $l$ and disjoint
from $\Rint M$. In particular, the line $l(u_2)$ is parallel to
$l$ and disjoint from $\Rint M$, contradicting the assumption
$\Int K \cap H(u_2) \ne \varnothing$. If $\delta(u_1) < 0$, we
similarly obtain a contradiction with $\Int K \cap H(-u_2) \ne
\varnothing$ by considering the counterclockwise bypass of $P \cap
S^{n-1}$ from $u_1$ to $-u_1$.

Let $v \ne o$. Denote by $C$ the circle in $P$ with diameter $[o,
v]$. Clearly, $\delta(u)  u \notin C$ for all $u \in P \cap
S^{n-1}$ due to the assumption that no line $l(u)$ contains $v$.
Considering separately the cases $C \cap l = \{ v \}$ and $C \cap
l \ne \{ v \}$, we obtain, similarly to the case $v = o$ above,
that $\Rint M \cap l(u_2) = \varnothing$ or $\Rint M \cap l(-u_2)
= \varnothing$, which is impossible.
\end{proof}

We recall that an $n$-dimensional closed convex set $K \subset
\R^n$ distinct from $\R^n$ is \textit{strictly convex} if $\Bd K$
contains no line segments; $K$ is called \textit{regular} provided
any point $x \in \Bd K$ belongs to a unique hyperplane supporting
$K$.

\begin{lemma} \label{section3}
If $K$ is neither strictly convex nor regular, then $\Bd K$ is a
sheet of elliptic cone.
\end{lemma}

\begin{proof}
First, we are going to show that if $K$ is not regular, then $K$
is not strictly convex. Indeed, suppose that $K$ is not regular
and choose a singular point $x \in \Bd K$. Let $G_1$ and $G_2$ be
distinct hyperplanes both supporting $K$ at $x$, and $G$ be a
hyperplane through $G_1 \cap G_2$ supporting $K$ and distinct from
both $G_1$ and $G_2$. Choose in $G$ an $(n-2)$-dimensional plane
$L$ through $x$ which is distinct from $G_1 \cap G_2$. By
Lemma~\ref{section2}, there is a hyperplane $H(u)$, $u \in
S^{n-1}$, containing $L$. Because $H(u)$ meets $\Int K$, the point
$x$ is singular for the $(n-1)$-dimensional convex surface $E(u) =
\Bd K \cap H(u)$. Hence $E(u)$ must be a sheet of elliptic cone.
Choosing a line segment in $E(u)$, we conclude that $K$ is not
strictly convex.

Now, assume that $K$ is not strictly convex and choose a line
segment $[x, z] \subset \Bd K$. By Lemma~\ref{section2}, there is
a hyperplane $H(u_0)$, $u_0 \in S^{n-1}$, containing the line
through $x$ and $z$. Since the $(n-1)$-dimensional convex quadric
$E(u_0) = \Bd K \cap H(u_0)$ is line-free and not strictly convex,
it should be a sheet of elliptic cone. Let $v$ be the apex of
$E(u_0)$. Denote by $h_1$ the halfline $[v, x)$ and choose another
halfline $h_2 = [v, w) \subset E(u_0)$ such that the 2-dimensional
plane through $h_1 \cup h_2$ intersects $\Int K$ (this is possible
since $H(u_0)$ meets $\Int K$). Let $G_2$ be a hyperplane
supporting $K$ with the property $h_2 \subset G_2$. By the above,
$h_1 \not\subset G_2$.

Choose a halfline $h$ with apex $v$ tangent to $K$ and so close to
$h_1$ that $h \not\subset G_2$. Let $G$ be a hyperplane through
$h$ that supports $K$. By Lemma~\ref{section2}, there is a
hyperplane $H(u)$, $u \in S^{n-1}$, that meets $\Int K$ and
contains $h$. Since the section $E(u) = \Bd K \cap H(u)$ is
bounded by both $G$ and $G_2$, the point $v$ is singular for
$E(u)$. As above, $E(u)$ is a sheet of elliptic cone. Hence $h
\subset \Bd K$. Varying $h$ and $h_2$, we obtain by the argument
above that every tangent halfline of $K$ at $v$ lies in $\Bd K$.
This shows that $K$ is a convex cone with apex $v$. Finally,
choose a hyperplane $H(u_1)$, $u_1 \in S^{n-1}$, that properly
intersects $K$ along a bounded set (this is possible since $K$ is
line-free). By the assumption, $\Bd K \cap H(u_1)$ is an
$(n-1)$-dimensional ellipsoid. So, $\Bd K$ is a sheet elliptic
cone with apex $v$ generated by $\Bd K \cap H(u_1)$.
\end{proof}

\begin{lemma} \label{section4}
Let $K$ be strictly convex and regular. There is a scalar $\rho >
0$ with the property that for any vector $u \in S^{n-1}$ and a
point $x \in \Bd K \cap H(u)$ there is a ball $B_\rho(z)$ of
radius $\rho$ centered at a point $z \in H(u)$ such that
$B_\rho(z) \cap H(u) \subset K \cap H(u)$ and $x \in \Bd
B_\rho(z)$.
\end{lemma}

\begin{proof}
Since $\delta (u)$ is continuous on $S^{n-1}$, the set $\Delta =
\{ \delta(u)  u \mid u \in S^{n-1} \}$ is compact.

Case 1. Assume that $K$ is bounded and denote by $d$ the diameter
of $K$. By a compactness argument, there is a scalar $\gamma > 0$
with the following property: if $G(u)$ is a hyperplane parallel to
$H(u)$ and supporting $K$, then the distance between $G(u)$ and
$H(u)$ is at least $\gamma$. Choose a point $x \in \Int K$ and a
closed ball $B_\mu (x) \subset K$ of radius $\mu > 0$. Fix a
vector $u \in S^{n-1}$ and consider the $(n-1)$-dimensional
ellipsoid $E(u) = \Bd K \cap H(u)$. Let $G(u)$ be one of the
hyperplanes parallel to $H(u)$ and supporting $K$ such that $x$
lies either in $H(u)$ or between $H(u)$ and $G(u)$. Choose a point
$v \in K \cap G(u)$ and denote by $C$ the convex cone with apex
$v$ generated by $B_\mu (x)$. Clearly, $F(u) = C \cap H(u)$ is an
$(n-1)$-dimensional ellipsoid that lies in $E(u)$. Because $\| x -
v \| \le d$ and the distance between $H(u)$ and $G(u)$ is at least
$\gamma$, there is a scalar $\varepsilon > 0$, not depending on
$u$, such that $F(u)$ contains an $(n-1)$-dimensional ball
$B_\varepsilon (w) \cap H(u)$, with $w \in H(u)$. Since the
diameter of $E(u)$ is less than or equal to $d$, there is a scalar
$\rho > 0$ that satisfies the condition of the lemma.

Case 2. Assume that $K$ is unbounded. Because $\Delta$ is compact,
we can choose a closed halfspace $V$ whose boundary hyperplane $G$
is so far from $\Delta$ that the following conditions are
satisfied:
\begin{itemize}

\item[a)] $K_1 = K \cap V$ is a convex body and $\Delta \subset
\Int V$,

\item[b)] any $(n-1)$-dimensional convex quadric $E(u) = \Bd K
\cap H(u)$, $u \in S^{n-1}$, has a vertex in $\Int V$,

\item[c)] any unbounded $(n-1)$-dimensional convex quadric $E(u) =
\Bd K \cap H(u)$, $u \in S^{n-1}$, intersects the relative
interior of $K \cap G$.

\end{itemize}

Fix a vector $u \in S^{n-1}$ and consider the $(n-1)$-dimensional
convex quadric $E(u) = \Bd K \cap H(u)$. If $E(u)$ is an
ellipsoid, then we proceed similarly to Case 1, by choosing a ball
$B_{\varepsilon_1} (z)$ of radius $\varepsilon_1 > 0$ centered at
a point $z \in H(u)$ such that $B_{\varepsilon_1} (z) \cap H(u)
\subset K_1 \cap H(u)$. So, the lemma holds provided $E(u)$ is
bounded.

Now suppose that $E(u)$ is unbounded. By condition b), the vertex
of $E(u)$ lies in $\Int K_1$. Denote by $r(u)$ the radius of the
largest ball $B_{r(u)} (z)$, $z \in H(u)$, such that $B_{r(u)} (z)
\cap H(u) \subset K \cap H(u)$ and the vertex of $E(u)$ belongs to
$\Bd B_{r(u)} (z)$. Clearly, any other point $x \in E(u)$ lies in
the relative boundary of an $(n - 1)$-dimensional ball from $K
\cap H(u)$ of radius $r(u)$. Assume, for contradiction, that the
conclusion of Lemma~\ref{section4} does not hold. Then we can find
a sequence of unit vectors $u_1, u_2, \dots \in S^{n-1}$ such that
the unbounded convex quadrics $E(u_1), E(u_2), \dots$ satisfy the
condition $r(u_k) \to 0$ as $k \to \infty$. Since the vertices of
$E(u_k)$ belong to $K_1$, we conclude, by a compactness argument,
the existence of a subsequence $E(u_{k_1}), E(u_{k_2}), \dots$,
that converges to a convex quadric $E(u)$ with $r(u) = 0$, which
is a sheet of elliptic cone. The last is impossible since $K$ is
strictly convex.
\end{proof}

\begin{lemma} \label{section5}
Let $K$ be strictly convex and regular. There are hyperplanes
$H(u_1)$ and $H(u_2)$, $u_1, u_2 \in S^{n-1}$, such that both
sections $\Bd K \cap H(u_1)$ and $\Bd K \cap H(u_2)$ are
$(n-1)$-dimensional ellipsoids whose intersection is an
$(n-2)$-dimensional ellipsoid.
\end{lemma}

\begin{proof}
Since $K$ is line-free, there is a 2-dimensional subspace $P$ such
that the orthogonal projection, $M$, of $K$ onto $P$ is a
line-free closed convex set (see, e.g.,~\cite{sol06}). Denote by
$\mathcal{F}$ the family of lines $l(u) = P \cap H(u)$, $u \in P
\cap S^{n-1}$, such that $l(u) \cap M$ is bounded. Let $l(u_0)$ be
one of these lines. Put $[v, w] = l(u_0) \cap M$. Clearly,
$l(u_0)$ cuts $M$ into 2-dimensional closed convex subsets, $M'$
and $M''$, at least one of them, say $M'$, being compact. If there
is a line $l(u) \in \mathcal{F}$ which intersects the open line
segment $]v, w[$, then the hyperplanes $H(u)$ and $H(u_0)$ have
the desired property. Assume that no line $l(u) \in \mathcal{F}$
intersects $]v, w[$. Then no line $l(u) \in \mathcal{F}$
intersects $\Rint M'$. Indeed, if a line $l(u_1) \in \mathcal{F}$
intersected $\Rint M'$, then, continuously rotating $u \in P \cap
S^{n-1}$ from the initial position $u_1$, we would find a line
$l(u_2)$ supporting $M$ at $v$ or at $w$ (which is impossible
since $\Int K \cap H(u_2) \ne \varnothing$). This argument shows
that $M''$ is unbounded, since otherwise we repeat the
consideration above for $M''$.

Rotating $u \in P \cap S^{n-1}$ counterclockwise from the initial
position $u_0$, we observe that the lines $l(u) \in \mathcal{F}$
cover the whole unbounded branch of $\Rbd M''$ with endpoint $v$.
Similarly, the lines $l(u) \in \mathcal{F}$ cover the second
unbounded branch of $\Rbd M''$, with endpoint $w$. This implies
the existence of lines $l(u_3), l(u_4) \in \mathcal{F}$ such that
the line segments $l(u_3) \cap K$ and $l(u_4) \cap K$ have a
common interior point. Clearly, the respective ellipsoids $\Bd K
\cap H(u_3)$ and $\Bd K \cap H(u_4)$ satisfy the conclusion of the
lemma.
\end{proof}

\begin{lemma} \label{section6}
Let $K$ be strictly convex and regular. If $\Bd K$ contains an
open piece of real quadric surface, then $\Bd K$ is a convex
quadric.
\end{lemma}

\begin{proof}
Let $A$ be an open piece of real quadric surface $Q \subset \R^n$
which lies in $\Bd K$. We state that $\Bd K \subset Q$. Assume,
for contradiction, that $\Bd K \not\subset Q$, and choose a
maximal (under inclusion) open piece $B$ of $\Bd K \cap Q$ that
contains $A$. Let $U_r(x) \subset \R^n$ be an open ball with
center $x \in B$ and radius $r > 0$ such that $\Bd K \cap U_r(x)
\subset B$. Continuously moving $x$ towards $\Bd K \setminus B$,
we find points $x_0 \in B$ and $z \in \Bd K \setminus B$ with the
property $\Bd K \cap U_r(x_0) \subset B$ and $\| x_0 - z \| = r$.

Let $G$ be the hyperplane through $z$ which supports $K$ ($G$ is
unique since $K$ is regular). Denote by $\mathcal{L}$ the family
of $(n-2)$-dimensional planes $L \subset G$ that contain $z$ and
are distinct from the $(n-2)$-dimensional plane $L_0 \subset G$
tangent to $U_r(x_0) \cap G$ at $z$. By Lemma~\ref{section2}, any
plane $L \in \mathcal{L}$ lies in a respective hyperplane
$H_L(u)$. Due to Lemma~\ref{section4}, there is a scalar $t > 0$
so small that the union of $(n-1)$-dimensional convex quadrics
$E_L(u) = \Bd K \cap H_L(u)$, $L \in \mathcal{L}$, is dense in the
surface $t$-neighborhood $\Bd K \cap U_t(z)$ of $z$. Clearly, each
$E_L(u)$ has a nontrivial strictly convex intersection with $B$.
Since $E_L(u)$ is a unique convex quadric containing $E_L(u) \cap
B$, we conclude that $E_L (u) \subset Q$. By continuity,
\[
\Bd K \cap U_t(z) \subset \Cl (\cup \{ E_L(u) \mid L \in
\mathcal{L} \}) \subset Q.
\]
Hence $\Bd K \cap U_t(z) \subset B$, contrary to the choice of $z
\in \Bd K \setminus B$. Thus $\Bd K \subset Q$. Because $\Int K$
is a convex component of $\R^n \setminus Q$, the surface $\Bd K$
is a convex quadric.
\end{proof}

\begin{lemma} \label{section7}
Let $E_1$ and $E_2$ be $(n-1)$-dimensional ellipsoids in $\R^n$,
$n \ge 3$, which lie, respectively, in hyperplanes $H_1$ and $H_2$
of $\R^n$ such that $E = E_1 \cap E_2$ is an $(n-2)$-dimensional
ellipsoid. For any point $v \in \R^n \setminus (H_1 \cup H_2)$,
there is a quadric surface $Q$ that contains $\{ v \} \cup E_1
\cup E_2$.
\end{lemma}

\proof Clearly, we can choose Cartesian coordinates in $\R^n$ such
that
\begin{align*}
E &= \{ (0, 0, \xi_3, \dots, \xi_n) \mid \xi_3^2 + \dots + \xi_n^2
= 1 \}, \\
E_1 &= \{ (\xi_1, 0, \xi_3, \dots, \xi_n) \mid (\xi_1 - \rho_1)^2
+ \xi_3^2 + \dots + \xi_n^2 = \rho_1^2 + 1 \}, \\
E_2 &= \{ (0, \xi_2, \xi_3, \dots, \xi_n) \mid (\xi_2 - \rho_2)^2
+ \xi_3^2 + \dots + \xi_n^2 = \rho_2^2 + 1 \},
\end{align*}
where $\rho_1 > 0$ and $\rho_2 > 0$. Then $H_1$ and $H_2$ are
described by the equations $\xi_2 = 0$ and $\xi_1 = 0$,
respectively. Consider the family of quadric surfaces $Q(\mu)
\subset \R^n$ given by
\[
\xi_1^2 + \dots + \xi_n^2 + 2 \mu \xi_1 \xi_2 - 2 \rho_1 \xi_1 - 2
\rho_2 \xi_2 - 1 = 0,
\]
where $\mu$ is a parameter. Obviously, $E_i = H_i \cap Q(\mu)$, $i
= 1, 2$. The point $v = (\nu_1, \dots, \nu_n)$ belongs to $\R^n
\setminus ( H_1 \cup H_2)$ if and only if $\nu_1 \nu_2 \ne 0$.
Then $v \in Q(\mu_0)$ for
\[
\mu_0 = (1 + 2 \rho_1 \nu_1 + 2 \rho_2 \nu_2 - \nu_1^2 - \dots -
\nu_n^2) / (2 \nu_1 \nu_2). \qedhere
\]

\begin{lemma} \label{section8}
If $K$ is strictly convex and regular, then $\Bd K$ contains an
open piece of quadric surface.
\end{lemma}

\begin{proof}
We proceed by induction on $n \, (\ge 3)$. Let $n = 3$. By
Lemma~\ref{section5}, there are planes $H(u_1)$ and $H(u_2)$ such
that both sections $E_1 = \Bd K \cap H(u_1)$ and $E_2 = \Bd K \cap
H(u_2)$ are ellipses, with precisely two points, say $v$ and $w$,
in common. The set $\Bd K \setminus (E_1 \cup E_2)$ consists of
four open pieces, at least three of them being bounded because $K$
is line-free. We choose any of these pieces if $K$ is bounded, and
choose the piece opposite to the unbounded one if $K$ is
unbounded. Denote by $\gG$ the chosen piece. Let $L$ be a plane
through $[v, w]$ that misses $\gG$ and is distinct from both
$H(u_1)$ and $H(u_2)$. Clearly, there is a neighborhood $\Omega
\subset \Bd K$ of $v$ such that for any point $z \in \gG \cap
\Omega$, the plane $L_z$ through $z$ parallel to $L$ intersects
each of the ellipses $E_1$ and $E_2$ at two distinct points.

Choose a point $z \in \gG \cap \Omega$ and denote by $P_z$ the
plane through $z$ that supports $K$ ($P_z$ is unique since $K$ is
regular), and by $l_z$ the line through $z$ parallel to $[v, w]$.
Let $\mF_\alpha$, $\alpha > 0$, be the family of planes through
$l_z$ forming with $L_z$ an angle of size $\alpha$ or less. By
continuity and Lemma~\ref{section4}, the neighborhood $\Omega$ and
the scalar $\alpha$ can be chosen so small that for any given
plane $M \in \mF_\alpha$, every plane $H(u)$ through the line $M
\cap P_z$ intersects each of the ellipses $E_1$ and $E_2$ at two
distinct points. By the same lemma, we can find a scalar $r > 0$
such that for any plane $H(u)$ trough $z$, the convex quadric
curve $\Bd K \cap H(u)$ intersects the closed curve $\Bd K \cap
S_r(z)$ at two points, where $S_r(z) \subset \R^3$ is the sphere
of radius $r$ centered at $z$.

Due to Lemma~\ref{section7}, there is a quadric surface $Q$
containing $\{ z \} \cup E_1 \cup E_2$. By the above, given a
plane $M \in \mF_\alpha$, every plane $H(u)$ through the line $M
\cap P_z$ intersects $\Bd K$ along an ellipse, which has five
points in $Q$ (namely, $z$ and two on each ellipse $E_i$, $i = 1,
2$). Since an ellipse is uniquely defined by five points in
general position, the ellipse $E(u) = \Bd K \cap H(u)$ lies in $Q$
for any choice of a plane $H(u)$ through the line $M \cap P_z$,
where $M \in \mF_\alpha$. This argument shows the existence of two
open ``triangular'' regions in $\Bd K \cap Q \cap U_r(z)$ which
have a common vertex $z$ and are bounded by a pair of planes $M_1,
M_2 \in \mF_\alpha$ (see the shaded sectors of $\Bd K \cap U_r(z)$
in the figure below). Hence the case $n = 3$ is proved.

Suppose that the inductive statement holds for all $m \le n - 1$,
$n \ge 4$, and let $K \subset \R^n$ be a line-free, strictly
convex and regular closed convex set of dimension $n$ that
satisfies the hypothesis of Theorem~\ref{convex-quadric-section}.
Since the case when $K$ is compact is proved in~\cite{b-g87}, we
may assume that $K$ is unbounded. Then the recession cone $\Rec K$
contains halflines and is line-free. Choose a halfline $h \subset
\Rec K$ with endpoint $o$ such that the $(n-1)$-dimensional
subspace $L \subset \R^n$ orthogonal to $h$ satisfies the
condition $L \cap \Rec K = \{ o \}$. Then any proper section of
$K$ by a hyperplane parallel to $L$ is bounded (see,
e.g.,~\cite{sol06}).

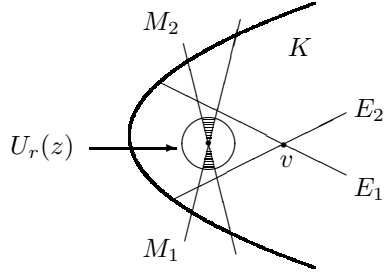
\begin{center}
\begin{picture}(240,110)

\put(101,15){\line(1,4){21}} \put(100.5,90){\line(1,-4){20}}
\put(97,31){\line(2,1){65}} \put(92,75){\line(2,-1){70}}
\put(110,52){\circle{18}} \put(110,52){\circle*{2}}
\put(138.5,51.5){\circle*{2}}

\put(65,50){\vector(1,0){33}} \put(140,85){$K$} \put(137,43){$v$}
\put(35,48){$U_r(z)$} \put(165,32){$E_1$} \put(165,62){$E_2$}
\put(85,9){$M_1$} \put(85,95){$M_2$}

\thicklines \qbezier(150,5)(10,55)(150,105)

\linethickness{0.01mm} \put(109,54){\line(1,0){2}}
\put(109,55){\line(1,0){2}} \put(109,56){\line(1,0){2}}
\put(108.5,57){\line(1,0){3}} \put(108.5,58){\line(1,0){3.5}}
\put(108.5,59){\line(1,0){3.5}} \put(108,60){\line(1,0){4}}
\put(108,61){\line(1,0){4.4}}

\put(109,49){\line(1,0){2}} \put(109,48){\line(1,0){2}}
\put(109,47){\line(1,0){2}} \put(108.5,46){\line(1,0){3}}
\put(108.5,45){\line(1,0){3.5}} \put(108.5,44){\line(1,0){3.5}}
\put(108,43){\line(1,0){4}}

\end{picture}
\end{center}

Because the set $\Delta = \{ \delta(u)  u \mid u \in L \cap
S^{n-1} \}$ is compact, we can choose a hyperplane $L_0$ parallel
to $L$ and properly intersecting $K$ so far from $\Delta$ that
every hyperplane $H(u)$, $u \in L \cap S^{n-1}$, intersects $\Rint
(K \cap L_0)$. Since any section $\Bd K \cap H(u) \cap L_0$, $u
\in L \cap S^{n-1}$, is an $(n-2)$-dimensional convex quadric, $K
\cap L_0$ satisfies the hypothesis of
Theorem~\ref{convex-quadric-section} (with $L_0$ instead of
$\R^n$). By the inductive assumption, $\Rbd (K \cap L_0)$ contains
a relatively open piece of an $(n-1)$-dimensional quadric, and
Lemma~\ref{section6} implies that $\Bd K \cap L_0$ is an
$(n-1)$-dimensional ellipsoid. Let $G \subset L_0$ be an
$(n-2)$-dimensional plane through the center of $K \cap L_0$. By
continuity and the argument above, there is an $\varepsilon > 0$
such that the hyperplanes $L_1$ and $L_2$ through $G$ forming with
$L_0$ an angle of size $\varepsilon$ also intersect $\Bd K$ along
$(n-1)$-dimensional ellipsoids $E_1$ and $E_2$, respectively.
Denote by $N$ the hyperplane through $G$ parallel to $h$, and
choose a point $v \in (\Bd K \cap N) \setminus (L_1 \cup L_2)$ so
close to $L_0$ that the hyperplane $L_0'$ through $v$ parallel to
$L_0$ satisfies the following conditions (see the figure below):
\begin{itemize}

\item[a)] $\Bd K \cap L_0'$ is an $(n-1)$-dimensional ellipsoid,

\item[b)] $L_0'$ intersects the relative interiors of both
$(n-1)$-dimensional solid ellipsoids $K \cap L_1$ and $K \cap
L_2$.

\end{itemize}

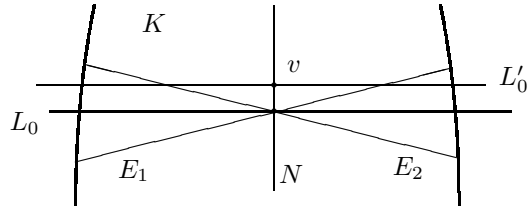
\begin{center}
\begin{picture}(220,75)

\put(25,35){\line(1,0){175}} \put(110,35){\circle*{2}}
\put(20,45){\line(1,0){170}} \put(110,45){\circle*{2}}
\put(110,5){\line(0,1){70}} \put(115,50){$v$} \put(10,28){$L_0$}
\put(195,45){$L_0'$} \put(60,65){$K$}

\put(35,16){\line(4,1){141.5}} \put(179.5,17.5){\line(-4,1){141}}
\put(51,10){$E_1$} \put(155,11){$E_2$}  \put(112,8){$N$}

\thicklines \qbezier(35,0)(35,40)(42,75)
\qbezier(180,0)(180,40)(173,75)

\end{picture}
\end{center}

By Lemma~\ref{section7}, there is a real quadric surface $Q$ that
contains $\{ v \} \cup E_1 \cup E_2$. Since the
$(n-1)$-dimensional ellipsoid $E_0' = \Bd K \cap L_0'$ is uniquely
determined by the set $\{ v \} \cup (E_1 \cap L_0') \cup (E_2 \cap
L_0')$, we have $E_0' \subset \Bd K \cap Q$. By continuity, there
is a $\beta > 0$ such that any hyperplane $L'$ through $G$ that
forms with $L_0'$ an angle of size $\beta$ or less satisfies
conditions a) and b) above; whence $\Bd K \cap L'$ is an
$(n-1)$-dimensional ellipsoid that lies in $\Bd K \cap Q$.
Clearly, the union of such ellipsoids $\Bd K \cap L'$ covers an
open piece of $Q$ that lies in $\Bd K$.
\end{proof}




\end{document}